\documentclass[10pt]{article}

\usepackage{enumerate}
\usepackage{amsmath}
\usepackage{amssymb,latexsym}
\usepackage{amsthm}
\usepackage{color}

\usepackage{graphicx}

\usepackage{hyperref}

\usepackage{amscd}

\DeclareMathOperator{\C}{\mathcal{C}}

\newtheorem{theorem}{Theorem}[section]

\newtheorem{lemma}[theorem]{Lemma}

\newtheorem{proposition}[theorem]{Proposition}
\newtheorem{result}[theorem]{Result}

\newcommand{\cG}{{\mathcal G}}

\newcommand{\cH}{{\mathcal H}}
\newcommand{\cC}{{\mathcal C}}
\newcommand{\cD}{{\mathcal D}}

\newcommand{\cL}{{\mathcal L}}

\newcommand{\PGaL}{{\mathrm P}\Gamma\mathrm{L}}
\newcommand{\N}{\mathrm{N}}
\newcommand{\F}{{\mathbb F}}

\newcommand{\la}{\langle}
\newcommand{\ra}{\rangle}

\newcommand{\PG}{\mathrm{PG}}

\newcommand{\ZG}{\mathcal{Z}(\mathrm{\Gamma L})}
\newcommand{\G}{\mathrm{\Gamma L}}

\title{Connections between scattered linear sets and MRD-codes}
\author{Olga Polverino and Ferdinando Zullo\thanks{
The
research  was supported by
 the Italian National
Group for Algebraic and Geometric Structures and their Applications (GNSAGA
- INdAM). The second author were also supported by the project ''VALERE: Vanvitelli pEr la RicErca" of the University of Campania ''Luigi Vanvitelli''. }}

\begin{document}
\maketitle

\begin{abstract}
The aim of this paper is to survey on the known results on maximum scattered linear sets and MRD-codes.
In particular, we investigate the link between these two areas. In \cite{Sheekey} Sheekey showed how maximum scattered linear sets of $\PG(1,q^n)$ define square MRD-codes.
Later in \cite{CSMPZ2016} maximum scattered linear sets in $\PG(r-1,q^n)$, $r>2$, were used to construct non square MRD-codes. Here, we point out a new relation regarding the other direction.
We also provide an alternative proof of the well-known Blokhuis-Lavrauw's bound for the rank of maximum scattered linear sets shown in \cite{BL2000}.
\end{abstract}

\section{Introduction}

Let $\Omega=\PG(V,\F_{q^n})=\PG(r-1,q^n)$, $q=p^h$, $p$ prime.
A set of points $L$ of $\Omega$ is called an $\F_q$-\emph{linear set} of $\Omega$ of rank $k$ if it consists of the points defined by the non-zero vectors of an $\F_q$-subspace $U$ of $V$ of dimension $k$, i.e.
\[ L=L_U=\{\langle \mathbf{u} \rangle_{\F_{q^n}} \colon \mathbf{u} \in U \setminus \{\mathbf{0}\}\}. \]

Linear sets are a generalization of subgeometries of projective spaces.
The term \emph{linear} has been used for the first time by Lunardon in \cite{Lu}, where he constructs special kind of blocking sets.
In recent years, linear sets have been intensively used to construct, classify or characterize many different objects like blocking sets, two-intersection sets, complete caps, translation spreads of the Cayley Generalized Hexagon, translation ovoids of polar spaces, semifield flocks, finite semifields and rank metric codes, see \cite{LVdV2015,Polverino,Sheekey} and the references therein.

\smallskip

Let $\Lambda=\PG(W,\F_{q^n})$ be a subspace of $\Omega$ and let $L_U$ be an $\F_q$-linear set of $\Omega$. Then $\Lambda \cap L_U=L_{W\cap U}$, and if $\dim_{\F_q} (W\cap U)=i$, i.e. if the $\F_q$-linear set $\Lambda \cap L_U=L_{W\cap U}$ has rank $i$, we say that $\Lambda$ has \emph{weight} $i$ in $L_U$, and we write $w_{L_U}(\Lambda)=i$.
Note that if $\Lambda$ has dimension $s$ and $L_U$ has rank $k$, then $0 \leq w_{L_U}(\Lambda) \leq \min\{k,n(s+1)\}$.
In particular, a point $P$ belongs to an $\F_q$-linear set $L_U$ if and only if $w_{L_U}(P)\geq 1$.
Also, we define the \emph{maximum field of linearity} of an $\F_q$-linear set $L_U$ as $\F_{q^\ell}$ if $\ell$ is the largest integer such that $\ell \mid n$ and $L_U$ is an $\F_{q^\ell}$-linear set.

\smallskip

One of the most natural questions about linear sets is their equivalence; especially in the applications it is crucial to have methods to establish whether two linear sets are equivalent or not.
Two linear sets $L_U$ and $L_W$ of $\Omega=\PG(r-1,q^n)=\PG(V,\F_{q^n})$ are said to be $\PGaL$-\emph{equivalent} (or simply \emph{projectively equivalent}) if there exists $\varphi \in \PGaL (r,q^n)$ such that $L_U^\varphi=\varphi(L_U)=L_W$.
If $U$ and $W$ are $\F_q$-subspaces of $V$ which are in the same $\G (r,q^n)$-orbit, then $L_U$ and $L_W$ are equivalent.
Indeed, if $f \in \G (r,q^n)$ and $f(U)=W$ then, denoting by $\varphi_f$ the semilinear collineation induced by $f$ (i.e. $\varphi_f(\langle \mathbf{u}\rangle_{\F_{q^n}})=\langle f(\mathbf{u})\rangle_{\F_{q^n}}$ ), then $\varphi_f(L_U)=L_{f(U)}=L_W$.
This is only a sufficient condition for the equivalence of two linear sets.
In general the $\G(r,q^n)$-orbit of an $\F_q$-subspace $U$ of $V$ does not determine the $\PGaL(r,q^n)$-orbit of the corresponding linear set $L_U$. If the $\G(r,q^n)$-orbit of $U$ completely determines the $\PGaL(r,q^n)$-orbit of $L_U$ we call $L_U$ \emph{simple}. More precisely, $L_U$ is a simple $\F_q$-linear set if for each $\F_q$-subspace $W$ of $V$ such that $\dim_{\F_q} U= \dim_{\F_q} W$ and $L_U=L_W$, the subspaces $U$ and $W$ are in the same $\G(r,q^n)$-orbit.
In \cite{CsMP} the authors investigated the equivalence problem between $\F_q$-linear sets of rank $n$ on the projective line $\PG(1,q^n)$.
The idea is to study first the $\Gamma\mathrm{L}(2,q^n)$-orbits of the subspaces defining the linear set and then to study the equivalence between two linear sets.
More precisely, they give the following definitions of $\mathcal{Z}(\Gamma\mathrm{L})$-class and $\Gamma\mathrm{L}$-class (see \cite[Definitions 2.4 \& 2.5]{CsMP}) of an $\F_q$-linear set of a line.

\smallskip

Let $L_U$ be an $\mathbb{F}_q-$linear set of $\PG(1,q^n)=\PG(V,\mathbb{F}_{q^n})$ of rank $n$ with maximum field of linearity $\mathbb{F}_q$.

We say that $L_U$ is of $\mathcal{Z}(\Gamma\mathrm{L})$-\emph{class} $r$ if $r$ is the greatest integer such that there exist $\F_q$-subspaces $U_1, U_2, \ldots, U_r$ of $V$ with $L_{U_i}=L_U$ for $i \in \{1,2,\ldots,r\}$ and $U_i \neq \lambda U_j$ for each $\lambda \in \F_{q^n}^*$ and for each $i \neq j$, $i, j \in \{1,2,\ldots,r\}$.
We say that $L_U$ is of $\Gamma\mathrm{L}$-\emph{class} $s$ if $s$ is the greatest integer such that there exist $\mathbb{F}_q$-subspaces $U_1,\ldots,U_s$ of $V$ with $L_{U_i}=L_U$ for $i \in \{1,\ldots,s\}$ and there is no $f \in \Gamma \mathrm{L}(2,q^n)$ such that $U_i=U_j^f$ for each $i\neq j$, $i,j \in \{1,2,\ldots,s\}$.

\smallskip

If $L_U$ is of $\Gamma \mathrm{L}$-class one, then $L_U$ is simple.
For $n\le4$, any linear set of rank $n$ in $\PG(1,q^n)$ is simple \cite[Theorem 4.5]{CsMP}.

The $\Gamma\mathrm{L}$-class of a linear set is a projective invariant (by \cite[Proposition 2.6]{CsMP}) and hence plays a crucial role in the study of linear sets up to equivalence.
See \cite{CsMP,CsMP2018,CsZ2016} for further details on the equivalence issue.

\smallskip

In Sections \ref{sec:scattered} and \ref{sec:MRD}, we will survey on maximum scattered linear sets and MRD-codes, dealing with links between them in Section \ref{sec:scatteredMRD}.
In \cite{CSMPZ2016}, the authors proved that starting from a maximum scattered linear set of $\PG(r-1,q^n)$ it is possible to construct a special type of MRD-code. Our main contribution is to prove the converse.
Also, we will provide an alternative proof of the Blokhuis-Lavrauw's bound for the rank of scattered linear sets.

\section{Scattered linear sets}\label{sec:scattered}

For an $\F_q$-linear set $L_U$ of rank $k$ in $\Omega=\PG(r-1,q^n)=\PG(V,\F_{q^n})$ the bound
\begin{equation}\label{cardlinearsets}
|L_U| \leq q^{k-1}+q^{k-2}+\ldots+1,
\end{equation}
holds true. Hence, $L_U$ is called \emph{scattered} if it achieves the bound \eqref{cardlinearsets},
or equivalently if all of its points have weight one. In this case, we will also say that $U$ is a {\em scattered subspace of $V$}.

A scattered $\F_q$--linear set $L_U$ of $\Omega$ with highest possible rank is a {\it maximum scattered $\F_q$--linear set} of $\Omega$ and $U$ is called a {\it maximum scattered subspace} of $V$; see \cite{BL2000}.
For a recent survey on the theory of scattered spaces in Galois Geometry and its applications see \cite{Lavrauw}.
Blokhuis and Lavrauw in \cite{BL2000} obtained the following result on the rank of a maximum scattered linear set.

\begin{theorem}\cite[Theorems 2.1 and 4.3]{BL2000}\label{boundscattered}
If $L_U$ is a maximum scattered $\F_q$-linear set of $\Omega=\PG(r-1,q^n)$ of rank $k$, then if $r$ is even
\[ k = \frac{rn}{2}, \]
otherwise
\[ \frac{rn-n}{2}\leq k \leq \frac{rn}{2}. \]
\end{theorem}

In Subsection \ref{BLbound} we will present a new proof of this bound based on the Singleton-like bound for rank metric codes.

\smallskip

If $r$ is even, explicit examples of scattered $\F_q$-linear sets of rank $\frac{rn}{2}$ can be found in \cite[Theorem 2.5.5]{LavrauwThesis}, where Lavrauw proves that the linear set defined by
\[ U=\{(x_1,x_2,\ldots,x_{r/2},x_1^q,x_2^q,\ldots,x_{r/2}^q) \colon x_i \in \F_{q^n}, i\in\{1,\ldots,r/2\}\}, \]
is a (maximum) scattered $\F_q$-linear set of rank $\frac{rn}{2}$ of $\PG(r-1,q^n)$.

\smallskip

A special class of maximum scattered linear sets when $r$ is even are those of pseudoregulus type.
They have been first introduced by Marino, Trombetti and the first author in \cite{MPT2007} and further generalized by Lavrauw and Van de Voorde in \cite{LVdV2013}; their name comes from the geometric construction of Freeman in \cite{Free}.

\smallskip

Let $L=L_U$ be a scattered $\F_q$-linear set of $\Omega=\PG(2t-1,q^n)$ of rank $tn$, $t,n \geq 2$.
We say that $L$ is of \emph{pseudoregulus type} if
\begin{enumerate}
  \item there exist $\displaystyle m=\frac{q^{nt}-1}{q^n-1}$ pairwise disjoint lines of $\Omega$, say $s_1,s_2,\ldots,s_m$ such that
        \[ \omega_L(s_i)=n \,\, \text{for each} \,\, i=1,\ldots,m; \]
  \item there exist exactly two $(t-1)$-dimensional subspaces $T_1$ and $T_2$ of $\Omega$ disjoint from $L$ such that $T_j \cap s_i \neq \emptyset$ for each $i=1,\ldots,m$ and $j=1,2$.
\end{enumerate}
The set of lines $\mathcal{P}_L$ is said the $\F_q$-\emph{pseudoregulus} (or simply \emph{pseudoregulus}) of $\Omega$ associated with $L$ and we call $T_1$ and $T_2$ the \emph{transversal spaces} of $\mathcal{P}_L$ (or \emph{transversal spaces} of $L$).

\smallskip

All maximum scattered linear sets of $\Omega=\PG(2t-1,q^3)$, with $t \geq 2$, are of pseudoregulus type and they are all equivalent under the action of the collineation group of $\Omega$ as proved in \cite[Theorem 5]{LVdV2010} for the case $t=1$, in \cite[Propositions 2.7 and 2.8]{MPT2007} for the case $t=2$ and \cite[Section 3 and Theorem 4]{LVdV2013} for the case $t\geq 3$. In \cite[Theorem 3.11]{LuMaPoTr2014}, the authors characterize the linear sets of pseudoregulus type in terms of the associated projected subgeometry and they also show how to construct them.

\begin{theorem}\cite[Theorem 3.5]{LuMaPoTr2014}
Let $T_1=\PG(U_1,\F_{q^n})$ and $T_2=\PG(U_2,\F_{q^n})$ be two disjoint $(t-1)$-subspaces of $\Omega=\PG(V,\F_{q^n})=\PG(2t-1,q^n)$, with $n>1$, and let $\phi_f$ be a semilinear collineation between $T_1$ and $T_2$ having as companion automorphism an element $\sigma \in \mathrm{Aut}(\F_{q^n})$ such that $\mathrm{Fix}(\sigma)=\F_q$.
Then, for each $\rho \in \F_{q^n}^*$, the set
\[ L_{\rho,f}=\{ \langle \mathbf{u}+\rho f(\mathbf{u}) \rangle_{\F_{q^n}} \colon \mathbf{u}\in U_1\setminus\{\mathbf{0}\} \} \]
is an $\F_q$-linear set of $\Omega$ of pseudoregulus type whose associated pseudoregulus is $\mathcal{P}_{L_{\rho,f}}=\{\langle P,P^{\phi_f}\rangle_{\F_{q^n}} \colon P \in T_1\}$, with transversal spaces $T_1$ and $T_2$.
\end{theorem}

In \cite[Theorem 3.12]{LuMaPoTr2014} the authors prove that linear sets of pseudoregulus type in $\Omega$ can be written as $L_{\rho,f}$ in previous theorem.
In case of the projective line, i.e. $\Omega=\PG(1,q^n)$, it has been proved in \cite[Section 4]{LuMaPoTr2014} and in \cite[Remark 2.2]{DD2014} that all the linear sets of pseudoregulus type in $\PG(1,q^n)$ are $\mathrm{PGL}(2,q^n)$-equivalent and hence we can define
them as all the linear sets which are $\mathrm{PGL}(2,q^n)$-equivalent to
\[ L=\{\langle (x,x^q) \rangle_{\F_{q^n}} \colon x \in \F_{q^n}^*\}. \]
Every linear set can be obtained as a projection of a suitable subgeometry of a suitable subspace as vertex, see \cite[Theorems 1 and 2]{LP:04}.
The projecting configurations which gives the linear set $L$ were described geometrically by Csajb\'ok and Zanella in \cite{CsZ20162}.

\smallskip

In the $r$ odd case the situation is slightly more complicated.
For $n=2$, Baer subgeometries $\PG(r-1,q)$ of $\PG(r-1,q^2)$ define scattered linear sets of rank $r$ for each $q$ and $r$, so they attain Blokhuis-Lavrauw's bound.
For $r=3$ and $n=4$, Ball, Blokhuis and Lavrauw in \cite[Theorem 3.1]{BBL2000} (see also \cite[Theorem 2.7.9]{LavrauwThesis}) prove that there exist $\alpha,\beta \in \F_{q^{12}}$ (without giving explicit conditions on $\alpha$ and $\beta$) such that the $\F_q$-subspace of $\F_{q^{12}}$ represented by the equation
\[ x^{q^6}+\alpha x^{q^3}+\beta x=0 \]
defines a scattered $\F_q$-linear set of rank $6$ in $\PG(2,q^4)$, so the bound is attained again.

Existence results have been proved for $n-1 \leq r$, $n$ even and $q>2$ in \cite[Theorem 4.4]{BL2000}. Explicit constructions of scattered linear sets of rank $rn/2$ in $\PG(r-1,q^n)$, $r$ odd, $n$ even, have been shown by Bartoli, Giulietti, Marino and the first author of this paper in \cite[Theorem 1.2]{BGMP2015} for infinitely many values of $r$, $n$ and $q$.
More precisely, they first find the following three families \cite[Theorems 2.2, 2.3 and 2.10]{BGMP2015} of maximum scattered $\F_q$-linear sets.
Let $n=2t$ and let $\PG(r-1,q^{2t})=\PG(\F_{q^{2rt}},\F_{q^{2t}})$.

\begin{itemize}
  \item Let $\omega$ be an element of $\F_{q^{2t}}\setminus \F_{q^t}$, then $\F_{q^{rt}}(\omega)=\F_{q^{2rt}}$. For any prime power $q\geq 2$ and any integer $t \geq 2$ with $\gcd(t,r)=1$, the $\F_q$-linear set defined by
      \[\{ax^{q^i}+x\omega \colon x \in \F_{q^{rt}}\}\]
      and satisfying the assumptions $\gcd(i,2t)=1$, $\gcd(i,rt)=r$ and $\N_{q^{rt}/q^r}(a) \notin \F_q$ is a scattered $\F_{q}$-linear set of $\PG(r-1,q^{2t})$ of rank $\displaystyle rt=\frac{rn}2$.
  \item Let $\omega$ be an element of $\F_{q^{2t}}\setminus \F_{q^t}$. For any prime power $q\equiv 1 \pmod{r}$ and any integer $t\geq 2$, the $\F_q$-linear set defined by
      \[\{ax^{q^i}+x\omega\colon x \in \F_{q^{rt}}\}\]
      and satisfying the assumptions $\gcd(i,2t)=\gcd(i,rt)=1$ and $(\N_{q^{rt}/q}(a))^{\frac{q-1}{r}}\neq 1$ is a scattered $\F_q$-linear set of $\PG(r-1,q^{2t})$ of rank $\displaystyle rt=\frac{rn}2$.
  \item Now, let $r=3$. For each integer $t\geq 2$, the $\F_2$-linear set defined by
        \[\{x^2+bx^{2^{2t+1}}+x\omega \colon x \in \F_{2^{3t}}\}\]
        with $b \in \F_{2^{3t}}^*$, $\N_{2^{3t}/2^t}(b)\neq 1$ and such that $x+bx^{2^{2t+1}-1}\notin\F_{2^t}$ for each $x \in \F_{2^{3t}}^*$, is a scattered $\F_2$-linear set of $\PG(2,2^{2t})$ of rank $\displaystyle 3t=\frac{rn}{2}$.
\end{itemize}

Therefore, by using \cite[Theorem 3.1]{BGMP2015} and by decomposing $V$ as a direct sum of a fixed number of $s$-subspaces, one can get the following result.

\begin{theorem}\cite[Theorem 1.2]{BGMP2015}\label{maxscattpar}
There exist examples of maximum scattered $\F_q$-linear sets in $\PG(r-1,q^n)$ of rank $\displaystyle \frac{rn}{2}$ in the following cases:
\begin{enumerate}
  \item $q=2$, for each odd $r\geq 3$ and even $n\geq 4$,
  \item for each $q\geq 2$, odd $r\geq 3$ and even $n\geq 4$ such that $\gcd(n,s)=1$, for some odd $s$ with $3 \leq s \leq n$,
  \item for each odd $r\geq 3$, even $n\geq 4$ and $q \equiv 1 \pmod{s}$ with $s$ odd and $3\leq s \leq n$.
\end{enumerate}
\end{theorem}

The previous result does not cover all cases. The only missing case is when $n$ is even, $r$ is odd, $6 \mid n$ and $q>2$, $q\not\equiv 1\pmod 3$. The authors in \cite{CSMPZ2016} construct $\F_q$-linear sets of rank $3n/2$ of $\PG(2,q^n)$, $n$ even, proving the sharpness of the bound also in the remaining open cases. The construction relies on the existence of non-scattered linear sets of rank $3t$ of $\PG(1,q^{3t})$ (with $t=n/2$) defined by a well-chosen binomial polynomial.
More precisely, they prove the following.

\begin{theorem}\cite[Theorem 2.4]{CSMPZ2016}
Let $\omega \in \F_{q^{2t}}\setminus \F_{q^t}$. For any prime power $q$ and any integer $t\geq 2$, there exist $a, b \in \F_{q^{3t}}^*$ and an integer $1 \leq i \leq 3t-1$ such that $\gcd(i,2t)=1$ and the $\F_q$-linear set $L_U$ of rank $3t$ of the projective plane $\PG(\F_{q^{6t}},\F_{q^{2t}})=\PG(2,q^{2t})$, where
\[U=\{ ax^{q^i}+bx^{q^{2t+i}}+\omega x \colon x \in \F_{q^{3t}}\},\]
is a scattered linear set.
\end{theorem}

As a consequence, for any integers $r, n\geq 2$, $rn$ even, and for any prime power $q\geq 2$ the rank of a maximum scattered $\F_q$-linear set of $\PG(r-1,q^n)$ is $rn/2$.

\smallskip

The projective line case attracted a lot of attention, especially because of its connection with MRD-codes that we will explore in the next section.

If the point $\langle (0,1) \rangle_{\F_{q^n}}$ is not contained in the linear set $L_U$ of rank $n$ of $\PG(1,q^n)$ (which we can always assume after a suitable projectivity), then $U=U_f:=\{(x,f(x))\colon x\in \F_{q^n}\}$ for some $q$-polynomial
$\displaystyle f(x)=\sum_{i=0}^{n-1}a_ix^{q^i}\in \F_{q^n}[x]$. In this case we will denote the associated linear set by $L_f$.

Up to $\Gamma\mathrm{L}(2,q^n)$-equivalence, the known non-equivalent maximum scattered $\F_q$-subspaces of $\F_{q^n}^2$ are
\begin{itemize}
\item $U_1:= \{(x,x^{q^s}) \colon x\in \F_{q^n}\}$, $1\leq s\leq n-1$, $\gcd(s,n)=1$, found in \cite{BL2000};
\item $U_2:= \{(x,\delta x^{q^s} + x^{q^{n-s}})\colon x\in \F_{q^n}\}$, $n\geq 4$, $\N_{q^n/q}(\delta)\notin \{0,1\}$, $q\neq 2$, $\gcd(s,n)=1$, found in \cite{LunPol2001} for $s=1$ and in \cite{Sheekey} for $s\neq 1$;
\item $U_3:= \{(x,\delta x^{q^s}+x^{q^{s+n/2}})\colon x\in \F_{q^{n}}\}$, $n\in \{6,8\}$, $\gcd(s,n/2)=1$, $\N_{q^n/q^{n/2}}(\delta) \notin \{0,1\}$, with some conditions on $\delta$ and $q$, found in \cite{CMPZ};
\item $U_{4}:=\{(x, x^q+x^{q^3}+\delta x^{q^5}) \colon x \in \F_{q^6}\}$, $q$ odd and $\delta^2+\delta=1$, see \cite{CsMZ2018} for $q\equiv 0,\pm1 \pmod{5}$, and \cite{MMZ} for the remaining congruences of $q$;
\item $U_5:=\{(x, h^{q-1}x^q-h^{q^2-1}x^{q^2}+x^{q^4}+x^{q^5}) \colon x \in \F_{q^6}\}$, $h\in\F_{q^6}$, $h^{q^3+1}=-1$ and $q$ odd, \cite{BZZ,ZZ}.
\end{itemize}

Note that $L_{U_1}$ is the linear set of pseudoregulus type in $\PG(1,q^n)$. Its $\mathcal{Z}(\Gamma\mathrm{L})$-class and its $\Gamma\mathrm{L}$-class are known.
Denote by $\varphi$ the Euler phi function. By determining the transversal spaces of the associated variety, Lavrauw, Sheekey and Zanella in \cite{LSZ} proved that its $\ZG$-class is $\varphi(n)$.
Csajb\'ok and Zanella in \cite{CsZ2016} proved that its $\Gamma\mathrm{L}$-class is $\frac{\varphi(n)}2$.
In particular, this means that to check the $\mathrm{PGL}(2,q^n)$-equivalence between an $\F_q$-linear set $L_U$ and $L_{U_1}$ in $\PG(1,q^n)$, we have to check whether $U$ is $\mathcal{Z}\mathrm{(GL)}(2,q^n)$-equivalent with one of the $\varphi(n)$ subspaces of the form $U_1$, where $\displaystyle \mathcal{Z}\mathrm{(GL)}(2,q^n)$ is the centre of $\mathrm{GL}(2,q^n)$, while to check the $\mathrm{P}\Gamma\mathrm{L}(2,q^n)$-equivalence between $L_U$ and $L_{U_1}$, we have to check whether $U$ is $\Gamma\mathrm{L}(2,q^n)$-equivalent with one of the $\displaystyle \frac{\varphi(n)}2$ subspaces of the form $U_1$, since the subspaces $\{(x,x^{q^s}) \colon x \in \F_{q^n}\}$ and $\{(x,x^{q^{n-s}}) \colon x \in \F_{q^n}\}$ are $\Gamma\mathrm{L}(2,q^n)$-equivalent.
In \cite[Theorem 3]{LunPol2001} it is proved that $L_{U_2}$ and $L_{U_1}$ are not $\mathrm{P}\Gamma \mathrm{L}(2,q^n)$-equivalent when $q>3$, $n\geq 4$ and $\delta \neq 0$.

\smallskip

More recently, extending the definition given in \cite{ShVdV}, in \cite{CsMPZ2019}, jointly with Csajb\'ok and Marino, we introduced the family of $h$-scattered linear sets.
A linear set $L_U$ of $\PG(r-1,q^n)=\PG(V,\F_{q^n})$ is called a $h$-\emph{scattered linear set} if $\langle L_U \rangle= \PG(r-1,q^n)$ and for each $(h-1)$-subspace $\Omega$ of $\PG(r-1,q^n)$ we have
\[ w_{L_{U}}(\Omega) \leq h. \]
In this case, we also say that $U$ is an \emph{$h$-scattered $\F_q$-subspace} of $V$.
When $h=r-1$ we obtain the scattered linear sets w.r.t. hyperplanes introduced by Sheekey and Van de Voorde in \cite{ShVdV}.
We prove that the rank $k$ of an $h$-scattered linear set of $\PG(r-1,q^n)$, if $k>r$, is at most $\frac{rn}{h+1}$, otherwise it is a canonical subgeometry of $\PG(r-1,q^n)$, generalizing the bound of Blokhuis and Lavrauw.
An $h$-scattered linear set $L_U$ of $\PG(r-1,q^n)=\PG(V,\F_{q^n})$ is called \emph{maximum $h$-scattered} (and $U$ is called \emph{maximum $h$-scattered $\F_q$-subspace of }$V$) if its rank is $\frac{rn}{h+1}$.
Also, we determine the spectrum of the weights of the hyperplanes of $\PG(r-1,q^n)$ w.r.t. a maximum $h$-scattered linear set, which, together with a new type of duality for linear sets, bring us to prove the existence of maximum $h$-scattered linear sets in $\PG(r-1,q^n)$ with some conditions on $r,n$ and $h$. It is currently an open question whether for each $r,n$ and $h$ such that $h+1\mid rn$ there exists a maximum $h$-scattered linear set in $\PG(r-1,q^n)$.

\section{MRD-codes}\label{sec:MRD}

In 1978 Delsarte in \cite{Delsarte} introduced rank metric codes as
$q$-analogs of the usual linear error correcting codes
endowed with Hamming distance.
He studied rank metric codes in terms of bilinear forms on two finite-dimensional vector spaces over a finite fields and he called \emph{Singleton systems} those known as maximum rank distance codes.
The set of $m \times n$ matrices $\F_q^{m\times n}$ over $\F_q$ is a rank metric $\F_q$-space
with rank metric distance defined by
\[d(A,B) = \mathrm{rk}\,(A-B)\]
for $A,B \in \F_q^{m\times n}$.
A subset $\C \subseteq \F_q^{m\times n}$ is called a \emph{rank metric code} (RM-code for short). The minimum distance of $\C$ is
\[d(\C) = \min_{{A,B \in \C},\ {A\neq B}} \{ d(A,B) \}.\]
Also, we say that $\C$ has parameters $(n,m,q;d)$, where $d$ is the minimum distance of $\C$.
When $\C$ is an $\F_q$-linear subspace of $\F_q^{m\times n}$, we say that $\C$ is an $\F_q$-\emph{linear} RM-code and its dimension $\dim_{\F_q}\C$ is defined to be the dimension of $\C$ as a subspace over $\F_q$.
In the same paper, Delsarte also showed that the
parameters of these codes must obey a Singleton-like bound:
let $\C$ be an RM-code of $\F_q^{m\times n}$ and let $d$ be its minimum distance, then
\begin{equation}\label{eq:SingletonBound}
|\C| \leq q^{\max\{m,n\}(\min\{m,n\}-d+1)}.
\end{equation}

When equality holds, we call $\C$ \emph{maximum rank distance} (\emph{MRD} for short) code.
Examples of MRD-codes were first found by Delsarte in \cite{Delsarte} and rediscovered by Gabidulin in \cite{Gabidulin}; we will show the known constructions in Subsection \ref{sec:exMRD}.

\medskip

Let $\mathcal{C}\subseteq \F_{q}^{m \times n}$ be a rank metric code, the \emph{adjoint code} of $\C$ is
\[ \C^\top =\{C^t \colon C \in \C\}. \]

Define the symmetric bilinear form $\langle\cdot,\cdot\rangle$ on $\F_q^{m \times n}$ by
\[ \langle M,N \rangle= \mathrm{Tr}(MN^t). \]
The \emph{Delsarte dual code} of an $\F_q$-linear RM-code $\C$ is
\[ \C^\perp = \{ N \in \F_q^{m\times n} \colon \langle M,N \rangle=0 \, \text{for each} \, M \in \C \}. \]

By using the machinery of association schemes, Delsarte in \cite{Delsarte} proved the following result.

\begin{lemma}\label{dualMRD}\cite[Theorem 5.5]{Delsarte}
Let $\C\subseteq \F_{q}^{m \times n}$ be an $\F_q$-linear MRD-code of dimension $k$ with $d>1$. Then the Delsarte dual code $\C^\perp\subseteq \F_{q}^{m\times n}$ is an MRD-code of dimension $mn-k$.
\end{lemma}

An elementary proof of the result above can be found in \cite{Ravagnani}.

\smallskip

Because of the classifications of the isometries of $\F_q^{m \times n}$ with $m,n \geq 2$ (see e.g. \cite[Theorem 3.4]{Wan}), if $m\neq n$ two RM-codes $\C$ and $\C'$ are \emph{equivalent} if and only if there exist $X \in \mathrm{GL}(m,q)$, $Y \in \mathrm{GL}(n,q)$, $Z \in \F_q^{m\times n}$ and a field automorphism $\sigma$ of $\F_q$ such that
\[\C'=\{XC^\sigma Y + Z \colon C \in \C\}.\]
If $m=n$ we have two possible definitions:
\begin{enumerate}
  \item $\C$ and $\C'$ are \emph{equivalent} if there exist invertible matrices $X,Y \in \F_q^{n\times n}$, $Z \in \F_q^{n \times n}$ and a field automorphism $\sigma$ of $\F_q$ such that
      \[ \C'=\{XC^\sigma Y + Z \colon C \in \C\}. \]
  \item $\C$ and $\C'$ are \emph{strongly equivalent} if there exist invertible matrices $A,B \in \F_q^{n\times n}$, $Z \in \F_q^{n \times n}$ and a field automorphism $\sigma$ of $\F_q$ such that
      \[\C'=\{XC^\sigma Y + Z \colon C \in \C\} \,\, \text{or} \,\, \C'=\{X(C^t)^\sigma Y + Z \colon C \in \C\}. \]
\end{enumerate}
Note that, if $\C$ is an RM-code then the set of all RM-codes strongly equivalent to $\C$ is the union of the set of all RM-codes equivalent to $\C$ and the set of all RM-codes equivalent to $\C^\top$.
When $\C$ and $\C'$ are $\F_q$-linear, we may always assume that $Z=0$. Indeed, for $C=0$ we get $Z \in \C'$ and hence
\[\C'-Z=\{ C'-Z \colon C'\in \C' \}=\C'.\]
For further details on the equivalence of RD-codes see also \cite{Berger,Morrison}.

\medskip

The weight of a codeword $C\in\C$ is the rank of the matrix
corresponding to $C$. The spectrum of weights of an MRD-code is
``complete'' in the following sense.
Denote by $A_i(\C)$ the number of codewords of  weight $i$ of an RM-code $\C$,
then the following holds.

\begin{lemma}\cite[Lemma 2.1]{LTZ2}\label{weight}
  Let $\C$ be an MRD-code in $\F_q^{m\times n}$ with minimum distance $d$ and
  suppose $m \leq n$. Assume that the null matrix $O$ is in $\cC$.
  Then, for any $0 \leq l \leq m-d$, we have $A_{d+l}(\C)>0$, i.e. there
  exists at least
  one matrix $C \in \C$ such that $\mathrm{rk} (C) = d + l$.
\end{lemma}

\smallskip

Delsarte in \cite{Delsarte} (and later Gabidulin in \cite{Gabidulin}) precisely determine the weight distribution of an MRD-code.

\begin{theorem}\label{weightdistribution}
Let $\cC$ be an MRD-code in $\F_q^{m\times n}$ with minimum distance $d$ and
suppose $m \leq n$. Then
\[ A_{d+l}(\C)={m \brack d+l}_q \sum_{t=0}^l (-1)^{t-l}{l+d \brack l-t}_q q^{\binom{l-t}{2}}(q^{n(t+1)}-1), \]
for $l \in \{0,1,\ldots,m-d\}$.
In particular, the number of codewords with minimum weight $d$ is
\begin{equation}\label{numberofmin}
A_d(\C)={m \brack d}_q(q^n-1)= \frac{(q^n-1)(q^m-1)(q^{m-1}-1)\cdots(q^{m-d+1}-1)}{(q^d-1)(q^{d-1}-1)\cdots(q-1)}.
\end{equation}
\end{theorem}

\medskip

In general, it is difficult to determine whether two RM-codes are equivalent or not.
\emph{Idealiser} are useful tools criterion to handle the equivalence issue.

\smallskip

Let $\C\subset \F_q^{m\times n}$ be an RM-code; its \emph{left} and \emph{right idealisers}
$L(\C)$ and $R(\C)$ are defined as
\[ L(\C)=\{ Y \in \F_q^{m \times m} \colon YC\in \C\hspace{0.1cm} \text{for all}\hspace{0.1cm} C \in \C\},\]
\[ R(\C)=\{ Z \in \F_q^{n \times n} \colon CZ\in \C\hspace{0.1cm} \text{for all}\hspace{0.1cm} C \in \C\}.\]
These notions have been introduced by  Liebhold and Nebe in \cite[Definition 3.1]{LN2016}.
Such sets appear also in the paper of Lunardon, Trombetti and Zhou \cite{LTZ2}, where they are
respectively called \emph{middle nucleus} and \emph{right nucleus}; therein the authors investigate these sets proving the following results.

\begin{result}\cite[Propositions 4.1 and 4.2, Theorem 5.4 \& Corollary 5.6]{LTZ2}\label{idealis}
If $\cC_1$ and $\cC_2$ are equivalent $\F_q$-linear RM-codes of $\F_q^{m\times n}$, then their left (resp. right) idealisers are also equivalent.
Let $\C$ be an $\F_q$-linear RM-code of $\F_q^{m \times n}$.
The following statements hold:
\begin{itemize}
  \item [(a)] $L(\C^\top)=R(\C)^\top$ and $R(\C^\top)=L(\C)^\top$;
  \item [(b)] $L(\C^\perp)=L(\C)^\top$ and $R(\C^\perp)=R(\C)^\top$.
\end{itemize}
Suppose that $\C$ is an $\F_q$-linear MRD-code of $\F_q^{m\times n}$ with minimum distance $d>1$.
If $m \leq n$, then $L(\C)$ is a finite field with $|L(\C)|\leq q^m$.
If $m \geq n$, then $R(\C)$ is a finite field with $|R(\C)|\leq q^n$.
In particular, when $m=n$ then $L(\C)$ and $R(\C)$ are both finite fields.
\end{result}

\medskip

In \cite{delaCruz} de la Cruz, Kiermaier, Wassermann and Willems point out the one-to-one correspondence between quasifields and MRD-codes of $\F_q^{n\times n}$ with minimum distance $n$.
For definitions and properties of quasifields and semifields we refer to \cite{Dembowski,LavrauwPolverino}.

\begin{theorem}\cite[Theorems 2, 3 \& 4]{delaCruz}\label{MRD-semifields}
If $\mathbb{K}$ is a finite field then
\begin{itemize}
  \item MRD-codes in $\mathbb{K}^{n\times n}$ (containing the zero and identity matrix) with minimum distance $n$ correspond to finite quasifields $\mathcal{Q}$ with $\mathbb{K} \leq \ker \mathcal{Q}$ and $\dim_{\mathbb{K}}\mathcal{Q}=n$.
  \item Additively closed MRD-codes (containing the identity matrix) in $\mathbb{K}^{n\times n}$ with minimum distance $n$ correspond to finite semifields $\mathcal{S}$ with $\mathbb{K} \leq \ker \mathcal{S}$ and $\dim_{\mathbb{K}}\mathcal{S}=n$.
  \item $\mathbb{K}$-linear MRD-codes (containing the identity matrix) in $\mathbb{K}^{n\times n}$ with minimum distance $n$ correspond to finite division algebras $\mathcal{A}$ over $\mathbb{K}$ where $\mathbb{K}\leq Z(\mathcal{A})$ and $\dim_{\mathbb{K}} \mathcal{A}=n$.
\end{itemize}
\end{theorem}

\subsection{Representation as linearized polynomials and known examples of MRD-codes}\label{sec:exMRD}

Any $\F_q$-linear rank metric code over $\F_q$ can
be equivalently defined  either as a subspace of matrices in $\F_{q}^{m\times n}$
or as a subspace of $\mathrm{Hom}(V_n,V_m)$, where $V_i$ is an $i$-dimensional $\F_q$-vector space. In the present section we
shall recall a special representation in terms of
linearized polynomials.

Consider two vector spaces $V_n$ and $V_m$ over $\F_q$ with dimension respectively $n$ and $m$, respectively. If $n\geq m$
we can always regard $V_m$ as a subspace of $V_n$ and identify
$\mathrm{Hom}(V_n,V_m)$ with the subspace of those $\varphi\in\mathrm{Hom}(V_n,V_n)$ with
$\mathrm{Im}(\varphi)\subseteq V_m$. Also, $V_n\cong\F_{q^n}$,
when $\F_{q^n}$ is considered as a $\F_q$-vector space of dimension $n$.
Let now  $\mathrm{Hom}_q(\F_{q^n}):=\mathrm{Hom}_q(\F_{q^n},\F_{q^n})$ be the set of all
$\F_q$--linear $\F_{q^n}\rightarrow \F_{q^n}$ maps.
It is well known that each element of $\mathrm{Hom}_q(\F_{q^n})$ can be
represented in a unique way as a $q$--polynomial over $\F_{q^n}$ modulo $x^{q^n}-x$;
see \cite{lidl_finite_1997}.
In other words, for any $\varphi\in\mathrm{Hom}_q(\F_{q^n})$ there is
an unique polynomial $f(x)$ of the form
\[ f(x):=\sum_{i=0}^{n-1} a_i x^{q^i}\]
with $a_i \in \F_{q^n}$ such that
\[ \forall x\in\F_{q^n}\colon \varphi(x)=f(x)=a_0x+a_1x^q+\cdots+a_{n-1}x^{q^{n-1}}. \]
The set $\tilde{\cL}_{n,q}$ of the $q$-polynomials over $\F_{q^n}$ with degree less than or equal to $q^{n-1}$ with the
usual sum and scalar multiplication  is
a vector space over $\F_{q^n}$.
When it is regarded as a
vector space over $\F_q$, its dimension is $n^2$
and it is isomorphic to $\F_q^{n\times n}$.
Actually, $\tilde{\cL}_{n,q}$ endowed with the product $\circ$ induced by the functional
composition in $\mathrm{Hom}_q(\F_{q^n})$ modulo $x^{q^n}-x$ is an algebra over $\F_{q}$.

Hence, it is possible to see that any $\F_q$-linear rank metric code might be regarded as a suitable $\F_q$-subspace of $\tilde{\cL}_{n,q}$.
So, the definitions given for rank metric codes in $\F_q^{m\times n}$ may be reformulated in the linearized polynomials framework.

Two $\F_q$-linear rank metric codes $\cC$ and $\cC'$ are equivalent if and only if
there exist two invertible $q$-polynomials  $h$ and $g$ and a field automorphism $\sigma$ such that
$\cC'=\{h \circ f^\sigma \circ g \colon f\in \cC\}$, where if $\displaystyle f(x):=\sum_{i=0}^{n-1}a_ix^{q^i}$ then $\displaystyle f^\sigma(x)=\sum_{i=0}^{n-1}a_i^\sigma x^{q^i}$.

The notion of Delsarte dual code can be written in terms of
$q$-polynomials as follows,
see for example \cite[Section 2]{LTZ}.
Let $b:\tilde{\cL}_{n,q}\times\tilde{\cL}_{n,q}\to\F_q$ be the bilinear form
given by
\[ b(f,g)=\mathrm{Tr}_{q^n/q}\left( \sum_{i=0}^{n-1} f_ig_i \right) \]
where $\displaystyle f(x)=\sum_{i=0}^{n-1} f_i x^{q^i}$ and $\displaystyle g(x)=\sum_{i=0}^{n-1} g_i x^{q^i} \in \F_{q^n}[x]$ and
we denote by $\mathrm{Tr}_{q^n/q}$  the trace function $\F_{q^n}\to\F_q$ trace function, that is, $\mathrm{Tr}_{q^n/q}(x)=\sum_{i=0}^{n-1}x^{q^i}$.
The Delsarte dual code $\cC^\perp$ of an $\F_q$-subspace $\cC$ of $\tilde{\cL}_{n,q}$ is
\[\cC^\perp = \{f \in \tilde{\cL}_{n,q} \colon b(f,g)=0, \hspace{0.1cm}\forall g \in \cC\}. \]
Recall that the \emph{adjoint} $\hat{f}$ of the linearized polynomial $\displaystyle f(x)=\sum_{i=0}^{n-1} a_ix^{q^i} \in \tilde{\mathcal{L}}_{n,q}$ with respect to the bilinear form $b$ is
\[ \hat{f}(x)=\sum_{i=0}^{n-1} a_i^{q^{n-i}}x^{q^{n-i}}. \]
So, the adjoint code $\cC^\top$ of a set of $q$-polynomials $\cC$ is
\[ \cC^\top= \{\hat{f} \colon f \in \cC\}\subseteq \tilde{\cL}_{n,q}. \]

Furthermore, the left and right idealisers of an $\F_q$-linear code
$\cC\subseteq\tilde{\cL}_{n,q}$ can be written as
\[L(\cC)=\{\varphi(x) \in \tilde{\cL}_{n,q} \colon \varphi \circ f \in \cC\, \text{for all} \, f \in \cC\};\]
\[R(\cC)=\{\varphi(x) \in \tilde{\cL}_{n,q} \colon f \circ \varphi \in \cC\, \text{for all} \, f \in \cC\}.\]

When $L(\C)$ (resp. $R(\C)$) is equal to $\mathcal{F}_n=\{\alpha x \colon \alpha \in \F_{q^n}\}$ we say that $\C$ is $\F_{q^n}$-\emph{linear on the left} (resp. \emph{right}) (or simply $\F_{q^n}$-\emph{linear} if it is clear from the context).
In the literature it is quite common to find the term $\F_{q^n}$-linear instead of $\F_{q^n}$-linear on the left.
Of course, recalling that $\widehat{f\circ g}=\hat{g}\circ \hat{f}$, if $\cC$ is $\F_{q^n}$-linear on the left, then $\cC^\top$ is $\F_{q^n}$-linear on the right.
The following result holds.

\begin{result}\cite[Theorem 6.1]{CMPZ}\cite[Theorem 2.2]{CsMPZh}\label{rightvectorspace}
Let $\C$ be an $\F_q$-linear MRD-code of dimension $nk$ with parameters $(n,n,q;n-k+1)$.
Then $L(\C)$ (resp. $R(\C)$) has maximum order $q^n$ if and only if there exists an MRD-code $\C'$ equivalent to $\C$ which is $\F_{q^n}$-linear on the left (resp. on the right).
\end{result}

\smallskip

Now, we are going to present the known maximum rank distance codes by using their representation as sets of linearized polynomials of $\tilde{\mathcal{L}}_{n,q}$.
In \cite{Delsarte}, Delsarte gives the first construction for linear MRD-codes (he calls such sets \emph{Singleton systems}), though from the perspective of bilinear forms. Few years later, Gabidulin in \cite[Section 4]{Gabidulin} presents the same class of MRD-codes by using linearized polynomials.
Although these codes have been originally discovered by Delsarte, they are called \emph{Gabidulin codes} and they can be written as follows
\[ \mathcal{G}_{k}=\{a_0x+a_1x^q+\ldots+a_{k-1}x^{q^{k-1}} \colon a_0,\ldots,a_{k-1}\in \F_{q^n}\}=\langle x,x^q,\ldots,x^{q^{k-1}} \rangle_{\F_{q^n}}, \]
with $k\leq n-1$ and it results to be $\F_{q^n}$-linear on the left and on the right.
Kshevetskiy and Gabidulin in \cite{kshevetskiy_new_2005} generalize the previous construction obtaining the so-called \emph{generalized Gabidulin codes}
\[\mathcal{G}_{k,s}=\langle x,x^{q^s},\ldots,x^{q^{s(k-1)}} \rangle_{\F_{q^n}},\]
with $\gcd(s,n)=1$ and $k\leq n-1$.
$\mathcal{G}_{k,s}$ is an $\F_{q}$-linear MRD-code with parameters $(n,n,q;n-k+1)$ and $L(\mathcal{G}_{k,s})=R(\mathcal{G}_{k,s})\simeq \F_{q^n}$, see \cite[Lemma 4.1 \& Theorem 4.5]{LN2016} and \cite[Theorem IV.4]{Morrison}.
Note that, as proved in \cite{Gabidulin,kshevetskiy_new_2005}, this family is closed by the Delsarte duality and by the adjoint operation, more precisely $\mathcal{G}_{k,s}^\perp$ is equivalent to $\mathcal{G}_{n-k,s}$ and $\mathcal{G}_{k,s}^\top$ is equivalent to itself.
Note that the corresponding quasifield of a generalized Gabidulin code with minimum distance $n$ and $k=1$ is a field because of their idealisers.
Furthermore, generalized Gabidulin codes have been characterized in \cite{H-TM}.

\smallskip

More recently, Sheekey in \cite{Sheekey} proves that the set
\[ \mathcal{H}_k(\eta,h)=\{a_0x+a_1x^q+\ldots+a_{k-1}x^{q^{k-1}}+a_0^{q^{h}}\eta x^{q^k} \colon a_i \in \F_{q^n}\}, \]
with $k\leq n-1$ and $\eta \in \F_{q^n}$ such that $\N_{q^n/q}(\eta)\neq (-1)^{nk}$, is an $\F_q$-linear MRD-code of dimension $nk$ with parameters $(n,n,q;n-k+1)$.
This code is known as \emph{twisted Gabidulin code}.
It is possible to replace $q$ by $q^s$ (cf. \cite[Remark 9]{Sheekey}), with $\gcd(s,n)=1$, obtaining that the set
\[ \mathcal{H}_{k,s}(\eta,h)=\{a_0x+a_1x^{q^s}+\ldots+a_{k-1}x^{q^{s(k-1)}}+a_0^{q^{sh}}\eta x^{q^{sk}} \colon a_i \in \F_{q^n}\}, \]
with $k\leq n-1$ and $\eta \in \F_{q^n}$ such that $\N_{q^n/q}(\eta)\neq (-1)^{nk}$, is an $\F_q$-linear MRD-code of dimension $nk$ with parameters $(n,n,q;n-k+1)$.
This code is called \emph{generalized twisted Gabidulin code}.
Lunardon, Trombetti and Zhou in \cite{LTZ} determined the automorphism group of the generalized twisted Gabidulin codes and, up to equivalence, they proved that the generalized Gabidulin codes and twisted Gabidulin codes are both proper subsets of this class.

Clearly, with $s=1$, $\mathcal{H}_{k,s}(\eta,h)$ is the twisted Gabidulin code $\mathcal{H}_{k}(\eta,h)$ and for $\eta=0$ it is exactly the generalized Gabidulin code $\mathcal{G}_{k,s}$.
With $k=1$ the quasifield associated with $\mathcal{H}_{1,s}(\eta,h)$ is the generalized twisted field, which is a presemifield found by Albert \cite{Albert}.
Also, the authors in \cite[Corollary 5.2]{LTZ} determine its left and right idealisers: if $\eta \neq 0$, then
\begin{equation}\label{leftrightidealH}
L(\mathcal{H}_{k,s}(\eta,h))\simeq\F_{q^{\gcd(n,h)}} \,\, \text{and} \,\, R(\mathcal{H}_{k,s}(\eta,h))\simeq\F_{q^{\gcd(n,sk-h)}}.
\end{equation}
As for the above family, the class of generalized twisted Gabidulin codes is closed by the Delsarte duality and by the adjoint operation, more precisely $\mathcal{H}_{k,s}(\eta,h)^\perp$ is equivalent to $\mathcal{H}_{n-k,s}(-\eta,n-h)$ and $\mathcal{H}_{k,s}(\eta,h)^\top$ is equivalent to $\mathcal{H}_{k,s}(1/\eta,sk-h)$,
\cite[Theorem 6]{Sheekey} and \cite[Propositions 4.2 \& 4.3]{LTZ}.
Moreover, in \cite{GiuZ} the MRD-code $\mathcal{H}_{k,s}(\eta,0)$ has been characterized in terms of intersections with some of its conjugates.

In \cite[Theorem 7]{Sheekey}, the author proves that $\mathcal{G}_{k,s}$ is equivalent to $\mathcal{H}_{k,1}(\eta,h)$ if and only if $k \in \{1,n-1\}$ and $h \in \{0,1\}$, while the equivalence between $\mathcal{H}_{k,s}(\eta,h)$ and $\mathcal{H}_{k,t}(\theta,g)$ has been completely answered in \cite[Thereom 4.4]{LTZ}.

\smallskip

A further generalization of twisted Gabidulin codes is due to Otal and \"Ozbudak in \cite{OzbudakOtal}, they prove that the set
\[ \mathcal{A}_{k,s,q_0}(\eta,h)=\{a_0x+a_1x^{q^s}+\ldots+a_{k-1}x^{q^{s(k-1)}}+\eta a_0^{q_0^h}x^{q^{sk}} \colon a_i \in \F_{q^n}\}, \]
with $\gcd(n,s)=1$, $q=q_0^u$, $k<n$ and $\eta \in \F_{q^n}$ such that $\N_{q^n/q_0}(\eta)\neq (-1)^{nku}$, is an $\F_{q_0}$-linear MRD-code of size $q^{nk}$ with parameters $(n,n,q;n-k+1)$.
They call this family \emph{additive generalized twisted Gabidulin codes} and very recently Sheekey in \cite{Sheekey2018} generalizes this further by looking at skew polynomial rings.

\smallskip

Trombetti and Zhou in \cite{TZ} find a new family of MRD-codes of $\tilde{\mathcal{L}}_{n,q}$, with $n$ even. More precisely, the set
\[ \mathcal{D}_{k,s}(\gamma)=\left\{ ax+c_1x^{q^s}+\ldots+c_{k-1}x^{q^{s(k-1)}}+\gamma b x^{q^{sk}} \colon c_i \in \F_{q^{n}}, a,b \in \F_{q^{\frac{n}2}}  \right\}, \]
with $\gcd(s,n)=1$ and $\gamma \in \F_{q^{n}}$ such that $\N_{q^{n}/q}(\gamma)$ is a non-square in $\F_q$, is an MRD-code of size $q^{nk}$ and parameters $(n,n,q;n-k+1)$.
Both its idealisers are isomorphic to $\F_{q^{\frac{n}2}}$.

\smallskip

Apart from the two infinite families of $\F_{q}$-linear MRD-codes $\F_{q^n}$-linear on the left, $\cG_{k,s}$ and $\cH_{k,s}(\eta,0)$, there are a few other examples, known for $n \in \{6,7,8\}$.

\begin{itemize}
  \item In \cite{CMPZ}, Csajb\'ok, Marino, and Zanella, jointly with the first author, prove the following results
        \begin{itemize}
             \item for $q>4$ it is always possible to find $\delta \in
                   \F_{q^2}$ such that the set
                   \[ \C_1=\langle x,\delta x^{q}+x^{q^4} \rangle_{\F_{q^6}} \]
                   is an MRD-code with parameters $(6,6,q;5)$, \cite[Theorem 7.1]{CMPZ}.
                   In \cite[Theorem 7.3]{PolverinoZullo2019}, explicit conditions on $\delta$, that guarantee that $\mathcal{C}_1$ is an MRD-code, are exhibited.
                   Its Delsarte dual code is equivalent to
                   \[ \cD_1=\langle x^{q}, x^{q^{2}}, x^{q^{4}},x-\delta^{q} x^{q^3} \rangle_{\F_{q^6}}, \]
                    whose parameters are $(6,6,q;3)$;
            \item for $q$ odd and $\delta \in \F_{q^8}$ such that
                  $\delta^2=-1$ the set
                  \[ \C_2=\langle x,\delta x^{q}+x^{q^{5}} \rangle_{\F_{q^8}} \]
                  is an MRD-code with parameters $(8,8,q;7)$, \cite[Theorem 7.2]{CMPZ}. Its Delsarte dual code is equivalent to
                  \[ \cD_2=\langle x^{q},x^{q^2},x^{q^3},x^{q^5},x^{q^6},x-\delta x^{q^4} \rangle_{\F_{q^8}}, \]
                  whose parameters are $(8,8,q;3)$.
        \end{itemize}
  \item In \cite{CsMZ2018} Csajb\'ok, Marino, jointly with the second author, prove that for $q$ odd, $q\equiv 0,\pm 1 \pmod{5}$ and $\delta^2+\delta=1$ the set
         \[ \cC_3=\langle x,x^{q}+x^{q^{3}}+\delta x^{q^5} \rangle_{\F_{q^6}} \]
         is an MRD-code with parameters $(6,6,q;5)$. Later, Marino, Montanucci and the second author in \cite{MMZ} prove that this holds true for each $q$ odd. Its Delsarte dual code is equivalent to
         \[ \cD_3=\langle x^q,x^{q^3},-x+x^{q^2}, \delta x - x^{q^4} \rangle_{\F_{q^6}}, \]
         whose parameters are $(6,6,q;3)$.
  \item In \cite{ZZ}, Zanella and the second author prove that for $q$ odd$, q \equiv 1 \pmod{4}$ and $q\leq 29$ the set
        \[ \cC_4=\langle x,x^q-x^{q^2}+x^{q^4}+x^{q^5} \rangle_{\F_{q^6}} \]
         is an MRD-code with parameters $(6,6,q;5)$. Its Delsarte dual code is equivalent to
         \[ \cD_4=\langle x^{q^3},x^q+x^{q^2}, x^q - x^{q^4}, x^q-x^{q^5} \rangle_{\F_{q^6}}, \]
         whose parameters are $(6,6,q;3)$.
  \item In \cite{BZZ}, Bartoli, Zanella and the second author generalize the aforementioned example, proving that for $h\in\F_{q^6}$, $h^{q^3+1}=-1$ and $q$ odd the set
        \[ \cC_4'=\langle x, h^{q-1}x^q-h^{q^2-1}x^{q^2}+x^{q^4}+x^{q^5}\rangle_{\F_{q^6}}\]
        is an MRD-code with parameters $(6,6,q;5)$.  Its Delsarte dual code is equivalent to
         \[ \cD_4'=\langle x^{q^3},h^{q^2}x^q+h^qx^{q^2}, x^q -h^{q-1} x^{q^4}, x^q-h^{q-1}x^{q^5} \rangle_{\F_{q^6}}, \]
         whose parameters are $(6,6,q;3)$.
  \item In \cite{CsMPZh}, Csajb\'ok, Marino and Zhou, jointly with the first author, prove the following results
        \begin{itemize}
          \item for $q$ odd and $\gcd(s,7)=1$ the set
                \[ \cC_5=\langle x,x^{q^s}, x^{q^{3s}} \rangle_{\F_{q^7}} \]
                is an MRD-code with parameters $(7,7,q;5)$, \cite[Theorem 3.3]{CsMPZh}. Its Delsarte dual code is equivalent to
                \[ \cD_5=\langle x,x^{q^{2s}},x^{q^{3s}},x^{q^{4s}} \rangle_{\F_{q^7}}, \]
                whose parameters are $(7,7,q;4)$;
          \item for $q \equiv 1 \pmod{3}$ and $\gcd(s,8)=1$ the set
                \[ \cC_6=\langle x,x^{q^s}, x^{q^{3s}} \rangle_{\F_{q^8}} \]
                is an MRD-code with parameters $(8,8,q;6)$, \cite[Theorem 3.5]{CsMPZh}. Its Delsarte dual code is equivalent to
                \[\cD_6=\langle x,x^{q^{2s}},x^{q^{3s}},x^{q^{4s}},x^{q^{5s}} \rangle_{\F_{q^8}}\]
                whose parameters are $(8,8,q;4)$.
        \end{itemize}
\end{itemize}

Note that the examples presented in \cite{CsMPZh} have maximal idealisers.

\medskip

Other examples of MRD-codes, but not of square type, and further investigations can be found in \cite{H-TM,Puch,SchmidtZhou}.
In particular, we point out that it is possible to produce MRD-codes in $\F_q^{m\times n}$, with $m\leq n$, from an MRD-code in $\F_q^{n\times n}$ (hence from $\tilde{\mathcal{L}}_{n,q}$) by puncturing; such a technique has been studied in \cite{ByrneRavagnani,CsS,Martinez}.

\smallskip

The first example of non linear MRD-code has been found by Cossidente, Marino and Pavese in \cite{CossMP}, which has been generalized by Durante and Siciliano in \cite{DurSic}. By using a more geometric approach Donati and Durante present in \cite{DonDur} a further generalization. Also, Otal and \"Ozbudak in \cite{OzbudakOtal2018}, in the framework of linearized polynomials, find a further example.

\smallskip

For a recent survey on MRD-codes see \cite{sheekey_newest_preprint,PhDthesis}.

\section{Connections between maximum scattered linear sets and MRD-codes}\label{sec:scatteredMRD}

Now we will discuss connections between linear sets and MRD-codes.
We divide the results into two subsections: first we analyze the known results for the square case and then we characterize non-square MRD-codes with some restrictions on the involved parameters.
Finally, we conclude this section by giving an alternative proof of the well-known Blokhuis-Lavrauw's bound using a completely different approach.

\subsection{Square case}

In \cite[Section 5]{Sheekey} Sheekey shows that maximum scattered $\F_q$-linear sets of $\PG(1,q^n)$ yield $\F_q$-linear
MRD-codes with parameters $(n,n,q;n-1)$.
More precisely, let $U_f=\{(x,f(x)) \colon x \in \F_{q^n}\}$ be an $\F_q$-subspace of $\F_{q^n}\times\F_{q^n}$ for some $q$-polynomial $f(x)$ over $\F_{q^n}$ and consider the following set of linearized polynomials over $\F_{q^n}$
\begin{equation}\label{Cf}
\C_f=\{af(x)+bx \colon a,b \in \F_{q^n}\}=\langle x, f(x) \rangle_{\F_{q^n}}.
\end{equation}
Note that $\cC_f$ is $\F_{q^n}$-linear on the left, i.e. $L(\C_f) \simeq \F_{q^n}$.

Then the following holds.

\begin{theorem}\cite{Sheekey}\label{4.1}
Let $f \in \tilde{\mathcal{L}}_{n,q}$.
Then $\C_f$ is an $\F_q$-linear MRD-code with parameters $(n,n,q;n-1)$ if and only if $U_f$ is a maximum scattered $\F_q$-subspace of $\F_{q^n}\times\F_{q^n}$.
\end{theorem}

Moreover, in \cite[Proposition 6.1]{CMPZ} the authors prove that if $\C$ is an MRD-code with parameters $(n,n,q;n-1)$ and with left idealiser isomorphic to $\F_{q^n}$, then $\C$ is equivalent to $\C_f$ (cf. \eqref{Cf}), for some $q$-polynomial $f$ and hence $\cC$ defines, by Theorem \ref{4.1}, a scattered $\F_q$-subspace of $\F_{q^n}\times\F_{q^n}$.

Also, we have the following result regarding the equivalence.

\begin{theorem}\label{equivMRD}\cite[Theorem 8]{Sheekey}
If $\C_f$ and $\C_g$ are two MRD-codes arising from maximum scattered subspaces $U_f$ and $U_g$ of $\F_{q^n}\times \F_{q^n}$, then $\C_f$ and $\C_g$ are equivalent if and only if $U_f$ and $U_g$ are  $\Gamma\mathrm{L}(2,q^n)$-equivalent.
\end{theorem}

Hence, if $\C_f$ and $\C_g$ are equivalent, it follows that the associated linear sets $L_{U_f}$ and $L_{U_g}$ are $\mathrm{P}\Gamma\mathrm{L}(2,q^n)$-equivalent.
The converse does not hold in general. Indeed, we may consider two non-equivalent generalized Gabidulin codes, namely $\mathcal{G}_{2,s}$ and $\mathcal{G}_{2,t}$, then $U_{x^{q^s}}$ and $U_{x^{q^t}}$ are not $\Gamma\mathrm{L}(2,q^n)$-equivalent but $L_{U_{x^{q^s}}}=L_{U_{x^{q^t}}}$.

\smallskip

More recently, Sheekey and Van de Voorde in \cite{ShVdV} extend Sheekey's construction.
They establish the correspondence between MRD-codes with parameters $(n,n,q;n-k)$ and $\F_{q^n}$-linear on the left, and scattered linear sets with respect to hyperplanes of $\PG(k-1,q^n)$  above defined, see \cite[Definition 14]{ShVdV}.
By \cite[Section 2.7]{Lunardon2017} and \cite{ShVdV} the next result follows, see also \cite[Result 4.7]{CsMPZ2019}.
\begin{result}
	\label{cod}
	$\cC$ is an $\F_q$-linear MRD-code of $\tilde{\cL}_{n,q}$ with minimum distance $n-k+1$ and with left-idealiser isomorphic to $\F_{q^n}$ if and only if up to equivalence
	\[\cC=\la f_1(x),\ldots,f_k(x)\ra_{\F_{q^n}}\]
	for some $f_1,f_2,\ldots,f_k \in \tilde{\cL}_{n,q}$	and the $\F_q$-subspace
	\[U_{\cC}=\{(f_1(x),\ldots,f_k(x)) \colon x\in \F_{q^n}\}\]
	is a maximum $(k-1)$-scattered $\F_q$-subspace of $\F_{q^n}^k$.
\end{result}

Sheekey and Van de Voorde generalize Theorem \ref{equivMRD} in \cite[Proposition 3.5]{ShVdV}.

\begin{proposition}
\label{resvdvs}
Let $\cC$ and $\cC'$ be two $\F_q$-linear MRD-codes of $\tilde{\cL}_{n,q}$ with minimum distance $n-k+1$ and with left-idealisers isomorphic to $\F_{q^n}$.
Then $U_{\cC}$ and $U_{\cC'}$ are $\G(k,q^n)$-equivalent if and only if $\cC$ and $\cC'$ are equivalent.
\end{proposition}

In contrast to the $k=2$ case, when $k>2$ the equivalence of MRD-codes coincides with the projective equivalence of the corresponding linear sets.

\begin{theorem}\cite[Theorem 4.10]{CsMPZ2019}
Let $\cC$ and $\cC'$ be two $\F_q$-linear MRD-codes of $\tilde {\cL}_{n,q}$ with minimum distance $n-k+1$, $k>2$, and with left-idealisers isomorphic to $\F_{q^n}$.
Then the linear sets $L_{U_{\cC}}$ and $L_{U_{\cC'}}$ are $\mathrm{P}\Gamma\mathrm{L}(k,q^n)$-equivalent
if and only if $\cC$ and $\cC'$ are equivalent.
\end{theorem}

\subsection{Non-square case}

We extend the link found by Sheekey in \cite{Sheekey} showing that MRD-codes of dimension $rn$ with parameters $\displaystyle \left(\frac{rn}2,n,q;n-1\right)$ and right idealiser isomorphic to $\F_{q^n}$ can be constructed from every maximum scattered $\F_q$--subspace of $V(r,q^n)$, $rn$ even and conversely.

In \cite[Theorem 3.2]{CSMPZ2016}, jointly with Csajb\'ok and Marino, we propose the following construction of rank metric codes starting from an $\F_q$-subspace $U$ of $V=V(r,q^n)$.

\bigskip

\textbf{Construction}\\
\noindent Let $U$ be an $\F_q$-subspace with dimension $\frac{rn}2$ of $V=V(r,q^n)$ where $rn$ is even.
Let $G\colon V \rightarrow W$, with $W=V(\frac{rn}2,q)$, be an $\F_q$-linear function such that $\ker G=U$.
For $\mathbf{v}\in V$, put
\[ \tau_{\mathbf{v}} \colon \lambda \in \F_{q^n} \mapsto \lambda \mathbf{v} \in V, \]
and also
\[ i:=\max\{\dim_{\F_q}(U\cap\langle\mathbf{v}\rangle_{\F_{q^n}}) \colon \mathbf{v}\in V\}. \]
If $i<n$ then
\begin{equation}\label{eq:CUG}
\mathcal{C}_{U,G}:=\{G\circ \tau_{\mathbf{v}} \colon \mathbf{v}\in V\}
\end{equation}
is an $\F_q$-linear rank metric code of dimension $rn$, with parameters $(rn/2,n,q;n-i)$ and with $R(\mathcal{C}_{U,G})=\mathcal{F}_n\simeq \F_{q^n}$.
Note that $G\circ \tau_{\mathbf{v}}\colon \F_{q^n} \rightarrow W$.
Therefore, $\cC_{U,G}$ is an MRD-code if and only if $i=1$, more precisely

\begin{theorem}\cite{CSMPZ2016}
\label{construction}
The rank metric code $\cC_{U,G}$ is an MRD-code if and only if $U$ is a maximum scattered $\F_q$-subspace of $V$.
\end{theorem}

We are now able to prove that each MRD-code with the same parameters as the above defined code produces a maximum scattered $\F_q$-subspace of $V=V(r,q^n)$.

\begin{theorem}\label{converse}
Let $\C$ be an $\F_q$-linear MRD-code with parameters $(t,n,q;n-1)$ with $t\geq n$ and $R(\C)\simeq \F_{q^n}$.
If $r=\dim_{R(\C)} \C$, then $rn$ is even, $\displaystyle t=\frac{rn}2$ and there exists a maximum scattered $\F_q$-subspace $U$ of $V(r,q^n)$ such that the code $\C_{U,G}$ defined in \eqref{eq:CUG} is equal to $\C$, for some $\F_q$-linear function $G$ as above.
\end{theorem}
\begin{proof}
We may assume, without loss of generality, that $\C$ is a set of $\F_q$-linear maps defined from $\F_{q^n}$ to $V(t,q)$ and, up to equivalence, by Result \ref{idealis} we may suppose that $R(\C)=\mathcal{F}_n$, i.e. for each $f \in \C$ and $\alpha \in\F_{q^n}$ it follows that
\[f\circ \tau_\alpha \in \C,\]
where $\tau_{\alpha}\colon x \in \F_{q^n} \mapsto \alpha x \in \F_{q^n}$.
Now, if $r$ is the dimension of $\C$ as a right vector space over $R(\C)$, then $|\C|=q^{nr}$ and hence $\dim_{\F_q} \C=nr$.
By the Singleton bound
\[ \dim_{\F_q} \C=nr = 2t, \]
then $nr$ is even and $\displaystyle t=\frac{rn}2$.
Since the minimum distance is $n-1$, for each $\F_q$-linear map $f\colon \F_{q^n} \rightarrow V(\frac{rn}2,q)$ in $\C\setminus \{0\}$ it results that
\[ \dim_{\F_q} \ker f \leq 1. \]
Now consider $\C=V(r,q^n)$ as a vector space over $\mathcal{F}_n$ with respect to the right composition  and let
\begin{equation}\label{subspace}
U=\{f \in \C \colon f(1)=0\}.
\end{equation}
We prove that $U$ is a scattered $\F_q$-subspace of $V(r,q^n)$.
Indeed, if $f \in U\setminus\{0\}$ and $g \in (\langle f\rangle_{\F_{q^n}}\cap U)\setminus\{0\}$, then $f(1)=0$ and $g=f\circ \tau_\alpha$ with $\alpha \in \F_{q^n}$. Hence $g(1)=f(\alpha)=0$, i.e. $1,\alpha \in \ker f$ and since $\dim_{\F_q} \ker f \leq 1$, it follows that $1$ and $\alpha$ are $\F_q$-dependent.
Hence $\dim_{\F_q} \langle f \rangle_{\F_{q^n}} \cap U=1$, i.e. $U$ is scattered.
Also, if $g \in \C$ with $\dim_{\F_q} \ker g =1$, then there exists $\lambda \in \F_{q^n}^*$ such that $g\circ \tau_{\lambda} \in U$.
Indeed, since $\dim_{\F_q} \ker g =1$ there exists $\lambda \in \F_{q^n}^*$ such that $g(\lambda)=g\circ \tau_{\lambda} (1)=0$.
Therefore, the set of all the elements of $\C$ with $1$-dimensional kernel is equal to the set $\displaystyle \bigcup_{f \in U\setminus\{0\}} \langle f\rangle_{\F_{q^n}}^*$,
where $\langle f\rangle_{\F_{q^n}}^*=\{f \circ \tau_\alpha \colon \alpha \in \F_{q^n}^*\}$.
Hence,
\[ A_{n-1}(\C)=(q^n-1)\frac{ |U^*|}{q-1},\]
where $A_{n-1}(\C)$ is the number of maps in $\C$ with rank $n-1$.
By evaluating $A_{n-1}$ using \eqref{numberofmin} of Theorem \ref{weightdistribution} applied to $\C^\top$, we get that
\[ A_{n-1}(\C)=(q^{\frac{rn}2}-1)\frac{q^n-1}{q-1} \]
and then $|U^*|=q^{\frac{rn}2}-1$, i.e. $U$ is a maximum scattered subspace of $\cC$.
Now, let
\[ G\colon g \in \cC \mapsto g(1) \in V\left(\frac{rn}2,q\right), \]
then $\ker G=U$ and
\[ \mathcal{C}_{U,G}=\{G\circ \tau_{\mathbf{v}} \colon \mathbf{v}\in V(r,q^n)\}=\cC. \]
Then the assert follows.
\end{proof}

Hence we have the following correspondence.

\begin{theorem}
Let $U$ be a scattered $\F_q$-subspace with dimension $rn/2$ of the $r$-dimensional $\F_{q^n}$-space $V=V(r,q^n)$, then the rank metric code $\cC_{U,G}$ (cf. \eqref{eq:CUG}) is an MRD-code of dimension $rn$, with parameters $(rn/2,n,q;n-1)$ and with $R(\C_{U,G})=\mathcal{F}_n\simeq \F_{q^n}$.
Conversely, if $\C$ is an $\F_q$-linear MRD-code with parameters $(t,n,q;n-1)$ with $t\geq n$, $R(\C)\simeq \F_{q^n}$ and $r=\dim_{R(\C)} \C$, then $rn$ is even, $\displaystyle t=\frac{rn}2$ and there exists a maximum scattered $\F_q$-subspace $U$ of $V(r,q^n)$ such that the code $\C_{U,G}$ defined in \eqref{eq:CUG} is equal to $\C$, for some $\F_q$-linear function $G$ as above.
\end{theorem}

\smallskip

We will show an example of the above construction.
Let $r$ be odd and $n=2t$. Some of the known families of maximum scattered $\F_q$-subspaces are given in the $r$-dimensional $\F_{q^{2t}}$-space $V=\F_{q^{2rt}}$ and they are of the form
\begin{equation}
U_f:=\{x\omega + f(x) \colon x\in \F_{q^{rt}}\},
\end{equation}
with $\omega \in \F_{q^{2t}} \setminus \F_{q^t}$ such that $\omega^2=\omega A_0+A_1$, $A_0,A_1\in\F_{q^t}$ and $f$ is an $\F_q$-linear transformation of $\F_{q^{rt}}$.
Then Construction \eqref{eq:CUG} gives  $\F_q$-linear MRD-codes with parameters $(rt,2t,q;2t-1)$.
Indeed, in this case $\{\omega,1\}$ is an $\F_{q^t}$-basis of $\F_{q^{2t}}$ and also an $\F_{q^{rt}}$-basis of $\F_{q^{2rt}}$, hence we can write any element $\lambda\in\F_{q^{2t}}$ as $\lambda= x \omega + y$, with $x, y \in\F_{q^t}$ and any element $v \in \F_{q^{2rt}}$ as $v=v_0\omega+v_1$ with $v_0,v_1 \in \F_{q^{rt}}$.\\
We fix $G \colon \F_{q^{2rt}} \rightarrow \F_{q^{rt}}$ as the map $x \omega + y \mapsto f(x)-y$.
For each $v=v_0 \omega + v_1 \in \F_{q^{2rt}}$ the map $\tau_v \colon \F_{q^{2t}} \rightarrow \F_{q^{2rt}}$ is as follows
\[\lambda=x \omega +y \mapsto v\lambda=x v_0 A_1 + y v_1 + \omega (x v_1 + y v_0 + x v_0 A_0),\]
and $\tau_v$ can be viewed as a function defined on $\F_{q^t}\times\F_{q^t}$.
Then the associated MRD-code consists of the maps $F_v=G\circ \tau_v$, i.e.
\[ F_v \colon (x,y)\in \F_{q^t}\times \F_{q^t} \mapsto f(x(v_1+v_0A_0)+yv_0)- x v_0 A_1 - y v_1.\]

\smallskip

For instance, put $f(x):=ax^{q^i}$,  $a\in\F_{q^{rt}}^*$, $1\leq i\leq rt-1$, $r$ odd.
For any $q\geq 2$ and any integer $t\geq 2$ with $\gcd(t,r)=1$, such that
\begin{enumerate}[(i)]
\item $\gcd(i,2t)=1$ and $\gcd(i,rt)=r$,
\item $\N_{q^{rt}/q^r}(a)\notin\F_q$,
\end{enumerate}
from Section \ref{sec:scattered} (see \cite[Theorem 2.2]{BGMP2015}), we get the $\F_q$-linear MRD-code with dimension $2rt$ and parameters $(rt,2t,q;2t-1)$:
\[ \{ F_v \colon v=\omega v_0+v_1,\, v_0, v_1 \in \F_{q^{rt}} \}, \]
where $F_v \colon \F_{q^t}\times\F_{q^t}\rightarrow\F_{q^{rt}}$ is defined by the rule
\begin{equation}\label{for:cod1}
F_v(x,y)=x^{q^i}a(A_0^{q^i}v_0^{q^i}+v_1^{q^i})-xA_1v_0+y^{q^i}av_0^{q^i}-yv_1.
\end{equation}
Note that, since $\gcd(i,rt)=r$, the above MRD-code is $\F_{q^r}$-linear as well, since for each $\mu\in\F_{q^r}$ and $v\in\F_{q^{2rt}}$ we have $\mu F_v=F_{\mu v}$.
Hence its left idealiser contains a subfield isomorphic to $\F_{q^r}$.

\subsection{An alternative proof of Blokhuis-Lavrauw's bound}\label{BLbound}

The proof of Theorem \ref{boundscattered} relies on the classical technique of double counting.
Here, by using the above methods (in particular the Singleton bound \eqref{eq:SingletonBound}), we can give an alternative proof of this bound.

\proof{(\emph{Blokhuis-Lavrauw's bound})}\label{BLbound}\\
\noindent Let $U$ be a scattered $\F_q$-subspace of dimension $k$ of $V=V(r,q^n)$, $r \geq 2$.
Clearly, $k \leq rn-n$, since $\displaystyle |L_U|=\frac{q^k-1}{q-1} \leq \frac{q^{rn}-1}{q^n-1}$.\\
\noindent Consider $W=V(rn-k,q)$ and
\[ G\colon V \rightarrow W\]
any surjective $\F_q$-linear map with $\ker G=U$. Consider the code
\[ \C=\mathcal{C}_{U,G}=\{G \circ \tau_{\mathbf{v}} \colon \mathbf{v} \in V\}. \]
The maps $G \circ \tau_{\mathbf{v}}$ of $\C$ are $\F_q$-linear maps from $\F_{q^n}$ to $W=V(rn-k,q)$ with kernel of dimension at most one.
Indeed, $\mu \in \ker (G \circ \tau_\mathbf{v})$ if and only if $G(\mu \mathbf{v})=0$, i.e. $\mu \mathbf{v} \in \ker G=U$; hence $\dim_{\F_q} \ker (G \circ \tau_\mathbf{v}) = \dim_{\F_q} (U \cap \langle \mathbf{v} \rangle_{\F_{q^n}})$
and $\dim_{\F_q} (U \cap \langle \mathbf{v} \rangle_{\F_{q^n}})\leq 1$ since $U$ is scattered.
Hence, the rank of $G \circ \tau_{\mathbf{v}}$ is greater than or equal to $n-1$, and so $\C$ is an $\F_q$-linear RM-code with parameters $(rn-k,n,q;n-1)$ where $rn-k \geq n$.
Also, since $U$ is scattered, $G\circ \tau_{\mathbf{v}}$ and $G\circ \tau_{\mathbf{w}}$ are equal if and only if $\mathbf{v}=\mathbf{w}$ and hence $|\C|=q^{nr}$.
Therefore, by the Singleton bound it follows that
\[nr \leq 2 (nr-k),\]
and hence $\displaystyle k \leq \frac{rn}2$, and the claim is proved.
\qed

\bigskip

\noindent Olga Polverino and Ferdinando Zullo\\
Dipartimento di Matematica e Fisica,\\
Universit\`a degli Studi della Campania ``Luigi Vanvitelli'',\\
I--\,81100 Caserta, Italy\\
{{\em \{olga.polverino,ferdinando.zullo\}@unicampania.it}}

\end{document}